\documentclass[a4paper,12pt]{article}

\unitlength\textwidth
\divide\unitlength by 150\relax

\usepackage{amsmath,amssymb}
\usepackage{bm}
\bmdefine{\NNN}{N}
\bmdefine{\ZZZ}{Z}
\bmdefine{\RRR}{R}
\bmdefine{\sss}{s}

\newcommand{\RRRRR}{{\mathcal R}}
\newcommand{\TTTTT}{{\mathcal T}}

\newcommand{\covers}{\mathrel{\cdot\!\!\!>}}
\newcommand{\covered}{\mathrel{<\!\!\!\cdot}}

\newcommand{\define}{\mathrel{:=}}
\newcommand{\definebycond}{\stackrel{\mathrm{def}}{\iff}}

\newcommand{\gor}{Gorenstein}
\newcommand{\cm}{Cohen-Macaulay}
\newcommand{\type}{{\rm type}}
\newcommand{\qed}{\nolinebreak\rule{.3em}{.6em}}
\newcommand{\joinirred}{join-irreducible}
\newcommand{\rank}{\mathrm{rank}}

\newcommand{\join}{\vee}

\newcommand{\relint}{{\rm{relint}}}

\newcommand{\condn}{{condition N}}
\newcommand{\scn}{{sequence with condition N}}
\newcommand{\sscn}{{sequences with condition N}}

\newcommand{\shseq}{{shortest sequence}}
\renewcommand{\shseq}{{irredundant sequence}}
\newcommand{\irseq}{{irredundant sequence}}
\newcommand{\rmax}{{r_{\max}}}
\newcommand{\mudown}{{\mu^{\downarrow}}}
\newcommand{\muup}{{\mu^{\uparrow}}}
\newcommand{\nudown}{{\nu^{\downarrow}}}
\newcommand{\nuup}{{\nu^{\uparrow}}}

\newtheorem{thm}{Theorem}[section]

\newtheorem{example}[thm]{Example}
\newtheorem{lemma}[thm]{Lemma}
\newtheorem{cor}[thm]{Corollary}
\newtheorem{definition}[thm]{Definition}
\newtheorem{prop}[thm]{Proposition}

\newtheorem{nndefinition}{Definition}
\newtheorem{nnremark}[nndefinition]{Remark}

\newcommand{\bigzerou}{\smash{\lower1.7ex\hbox{\bg 0}}}

\newcommand{\bigastu}{\smash{\lower1.7ex\hbox{\bg *}}}

\newcommand{\refeq}[1]{(\ref{#1})}
\numberwithin{equation}{section}

\newcommand{\mylabel}[1]{{\label{#1}\tt [#1]}}
\let\mylabel=\label

\title{%
On the generators of the canonical module of a Hibi ring:
a criterion of level property 
and the degrees of 
generators}

\author{Mitsuhiro MIYAZAKI\footnote{%
The author is supported partially by 
JSPS KAKENHI Grant Number 15K04818.}%
}
\date{Department of Mathematics, Kyoto University of Education,
1 Fujinomori, Fukakusa, Fushimi-ku, Kyoto, 612-8522, Japan}

\begin{document}


\maketitle

\sloppy

\begin{abstract}
In this paper, we study the minimal generating system of the canonical module of a Hibi
ring.
Using the results, we state a characterization of a Hibi ring to be level.
We also give a characterization of a Hibi ring to be of type 2.
Further, we show that the degrees of the elements of the minimal generating system
of the canonical module of a Hibi ring form a set of consecutive integers.
\\
Hibi ring, canonical module, level ring, \cm\ type
\\
MSC:
13F50, 13H10, 06A07
\end{abstract}

\section{Introduction}

A Hibi ring is an algebra with straightening law 
on a distributive lattice with many good properties.
It appears in many scenes of the study of 
rings with combinatorial structure.
In many important cases, the initial subalgebra of a subalgebra of a polynomial ring
has a structure of a Hibi ring.
A Hibi ring is a normal affine semigroup ring and is therefore
\cm\ by the result of Hochster \cite{hoc}.
Moreover there is a good description of the canonical module of it \cite{hib}.
However, as the Example of \cite[\S 1 e)]{hib} shows, a Hibi ring is
not necessarily level.

The present author gave a sufficient condition for a Hibi ring
to be level \cite{miy} and showed that the homogeneous coordinate ring
of a Schubert subvariety of a Grassmannian is level.
However, the sufficient condition given in \cite{miy} is far from necessary.


In this paper we analyze the generating system of the 
canonical module of a Hibi ring by using the description
of the canonical module of a normal affine semigroup ring by Stanley
\cite{sta2}.
Using this, we give a combinatorial
necessary and sufficient condition for a Hibi ring to be level.
We also show that the set of degrees of the 
generators of the canonical module
of a Hibi ring form a set of consecutive integers.
Here we call a member of the minimal generating system a generator for simplicity.

Further, we give a combinatorial criterion that the Hibi ring under consideration has \cm\ type 2.
Recall that a Hibi ring is \gor\ i.e., type of it is 1 if and only if the set of \joinirred\
elements of the distributive lattice defining that Hibi ring is pure \cite[\S3 d)]{hib}.
Since type 2 is next to type 1, 
we think it worth to give a criterion of a Hibi ring has type 2.

This paper is organized as follows.
Section \ref{sec:pre}  is a preparation for the main argument.
In Section \ref{sec:pre}, we recall some basic facts on Hibi rings, normal affine semigroup rings
and levelness of standard graded algebras.
Also,
we state some basic facts on posets.

In Section \ref{sec:level}, we state a necessary and sufficient condition
that a Hibi ring is level by considering the degrees of the generators of the
canonical module of the Hibi ring.
The key notion is a sequence with ``\condn'' (see Definition \ref{def:condn}).
We named the condition as \condn\ from the shape of the Hasse diagram  related to the
elements in the sequence.
We show that from a generator of the canonical module of a Hibi ring, one can 
construct a \scn\ 
with additional condition
in the poset of \joinirred\ elements corresponding to the Hibi ring.
We also show a kind of converse of this fact.
Thus, the range of the degrees of the generators of the canonical module of a Hibi ring can be 
described by \sscn.
Using this, we obtain a combinatorial criterion for a Hibi ring to be level.

Further, we show that the degrees of the elements of the minimal generating system of the
canonical module of a Hibi ring form a consecutive integers,
i.e.,
if there are elements of degrees $d_1$ and $d_2$ in the minimal system of generators of the
canonical module of a Hibi ring with $d_1< d_2$,
 then there exists an element with degree $d$ in the
minimal system of the generators of the canonical module of the Hibi ring for any 
integer $d$ with $d_1\leq d\leq d_2$.

In Sections \ref{sec:level type 2} and \ref{sec:non level type 2}, 
we state a combinatorial criterion that a Hibi ring
is of type 2.
In Section \ref{sec:level type 2}, we state a necessary and sufficient
condition that a Hibi ring is level and has type 2 and in Section \ref{sec:non level type 2}, 
we state a necessary and sufficient condition that a Hibi ring is not level and has type 2.


\section{Preliminaries}
\mylabel{sec:pre}

In this paper, all rings and algebras are assumed to be commutative with
identity element.
We denote by $\NNN$ the set of non-negative integers, by
$\ZZZ$ the set of integers, by
$\RRR$ the set of real numbers and
by $\RRR_{\geq0}$ the set of non-negative real numbers.

First we recall some definitions concerning finite partially
ordered sets (poset for short).
\begin{nndefinition}\rm
Let $Q$ be a finite poset.
\begin{itemize}
\item
A chain in $Q$ is a totally ordered subset of $Q$.
\item
For a chain $X$ in $Q$, we define the length of $X$ as $\#X-1$,
where $\#X$ denotes the cardinality of $X$.
\item
The maximum length of chains in $Q$ is called the rank of $Q$ and denoted as $\rank Q$.
\item
If every maximal chain of $Q$ has the same length, we say that $Q$ is pure.
\item
If $x$, $y\in Q$, $x<y$ and there is no $z\in Q$ with $x<z<y$,
we say that $y$ covers $x$ and denote 
$x\covered y$ or $y\covers x$.
\item
For $x$, $y\in Q$ with $x\leq y$, we set
$[x,y]_Q\define\{z\in Q\mid x\leq z\leq y\}$,
 and for $x$, $y\in Q$ with $x<y$, we set
$[x,y)_Q\define\{z\in Q\mid x\leq z< y\}$, 
$(x,y]_Q\define\{z\in Q\mid x< z\leq y\}$ and
$(x,y)_Q\define\{z\in Q\mid x< z< y\}$.
We denote $[x,y]_Q$ ($[x,y)_Q$, $(x,y]_Q$ $(x,y)_Q$ resp.) as $[x,y]$ ($[x,y)$, $(x,y]$, $(x,y)$
 resp.)
if there is no fear of confusion.
\item
Let $\infty$ be a new element which is not contained in $Q$.
We set $Q^+\define Q\cup\{\infty\}$ with the order
$x<\infty$ for any $x\in Q$.
\item
If $I\subset Q$ and
$x\in I$, $y\in Q$, $y\leq x\Rightarrow y\in I$,
then we say that $I$ is a poset ideal of $Q$.
\end{itemize}
\end{nndefinition}

Next we recall the definition and basic facts of Hibi rings.
Let 
$H$ be a finite distributive lattice and
$x_0$ the unique minimal element of $H$.
Let
$P$ be the set of \joinirred\ elements of $H$,
i.e., $P=\{\alpha\in H\mid \alpha=\beta\join\gamma\Rightarrow \alpha=\beta$ or $\alpha=\gamma\}$.
Note that we treat $x_0$ as a \joinirred\ element of $H$.

Then it is known that $H$ is isomorphic to $J(P)\setminus\{\emptyset\}$, 
ordered by inclusion, by the correspondence
\[
\begin{array}{ll}
\alpha\mapsto
	\{x\in P\mid x\leq\alpha\text{ in $H$}\}&\text{for $\alpha\in H$ and}
\\
I\mapsto
	\displaystyle\bigvee_{x\in I}x&\text{for $I\in J(P)\setminus\{\emptyset\}$,}
\end{array}
\]
where $J(P)$ is the set of poset ideals of $P$.
In particular, if $P$ is a poset with unique minimal element,
then there is a finite distributive lattice whose set of \joinirred\
elements is $P$.

Let $\{T_x\}_{x\in P}$ be a family of indeterminates indexed by $P$
and $K$ a field.
\begin{nndefinition}[\cite{hib}]
\rm
$\RRRRR_K(H)\define K[\prod_{x\leq \alpha}T_x\mid\alpha\in H]$.
\end{nndefinition}
$\RRRRR_K(H)$ is called the Hibi ring over $K$ on $H$.
We set $\deg T_{x_0}=1$ and $\deg T_{x}=0$ for any $x\in P\setminus\{x_0\}$.
Then, $\RRRRR_K(H)$ is a standard graded $K$-algebra, i.e.,
a non-negatively graded $K$-algebra whose degree $0$ part is $K$ and is generated over $K$
by the degree $1$ elements.
\begin{nndefinition}[\cite{hib}]\rm
For a map $\nu\colon P\to\NNN$, we set $T^\nu\define\prod_{x\in P}T_x^{\nu(x)}$.
We set
$\overline\TTTTT(P)\define\{\nu\colon P\to\NNN\mid x\leq y\Rightarrow \nu(x)\geq\nu(y)\}$
and
$\TTTTT(P)\define\{\nu\colon P\to\NNN\mid \nu(z)>0$ for any $z\in P$ and
$x< y\Rightarrow \nu(x)>\nu(y)\}$.
For $\nu\in\overline\TTTTT(P)$, we extend $\nu$ as a map from $P^+$ to $\NNN$ 
by setting $\nu(\infty)=0$.
\end{nndefinition}
With this notation,
\begin{thm}[\cite{hib}]
$\RRRRR_K(H)=\bigoplus_{\nu\in\overline\TTTTT(P)} KT^\nu$.
In particular, 
by the result of Hochster \cite{hoc},
$\RRRRR_K(H)$ is a normal affine semigroup ring and is \cm.
\end{thm}
Note that $\deg T^\nu=\nu(x_0)$.

Here we recall the description of the canonical module of a 
normal affine semigroup ring by Stanley.
\begin{thm}[{\cite[p.\ 81]{sta2}}]
Let $S$ be a finitely generated additive  submonoid of $\NNN^n$
and $X_1$, \ldots, $X_n$ indeterminates.
If the affine semigroup ring $\bigoplus_{s\in S}KX^s$ 
in $K[X_1, \ldots, X_n]$
is normal, then
the canonical module of $\bigoplus_{s\in S}KX^s$ is
$\bigoplus_{s\in S\cap \relint \RRR_{\geq0}S}KX^s$,
where $\relint Q$ denotes the interior of $Q$ in the affine space spanned by $Q$.
\end{thm}
By this result, the canonical module of a normal affine semigroup ring has the unique
minimal fine graded generating system up to non-zero scalar multiplication.
Therefore, we call a member of the minimal fine graded generating system
a generator for simplicity.

By applying this theorem to $\RRRRR_K(H)$, we see the following
\begin{cor}
The canonical module of $\RRRRR_K(H)$ is
$\bigoplus_{\nu\in\TTTTT(P)}KT^\nu$.
\end{cor}
In order to describe the generators of the canonical module
of $\RRRRR_K(H)$, we state the following
\begin{nndefinition}\rm
We define the order on $\TTTTT(P)$ by setting
$
\nu\leq\nu'\definebycond\nu'-\nu\in\overline\TTTTT(P)
$
for $\nu$, $\nu'\in\TTTTT(P)$,
where $(\nu'-\nu)(x)\define \nu'(x)-\nu(x)$.
\end{nndefinition}
Then the following fact is easily verified.
\begin{cor}
\mylabel{cor:can gen}
Let $\nu$ be an element of $\TTTTT(P)$.
Then $T^\nu$ is a generator of the canonical module of $\RRRRR_K(H)$
if and only if $\nu$ is a minimal element of $\TTTTT(P)$.
\end{cor}
Finally, we recall the following characterization of \gor\ 
property of $\RRRRR_K(H)$ by Hibi.
\begin{thm}[{\cite[\S 3 d)]{hib}}]
\mylabel{thm:gor case}
$\RRRRR_K(H)$ is \gor\ if and only if $P$ is pure.
\end{thm}

%
%
Here we recall that
Stanley \cite{sta8} defined the level property for a standard 
graded \cm\  algebra over a field and showed that a 
standard graded \cm\   algebra $A$  over a field is 
level if and only if the degree of the generators of the 
canonical module of $A$ is constant.

\medskip

In the following, let 
$H$ be a finite distributive lattice with unique minimal element $x_0$ and
let $P$ be the set of \joinirred\ elements of $H$.
We set $r\define\rank P^+$.
Since $\deg T^\nu=\nu(x_0)$ for $\nu\in\overline\TTTTT(P)$, 
and $\nu_0\colon P \to \NNN$,
$x\mapsto\rank[x,\infty]$
is a minimal element of $\TTTTT(P)$,
we see by 
Corollary \ref{cor:can gen} the following
\begin{cor}
\mylabel{cor:cri level}
$\RRRRR_K(H)$ is level if and only if 
$\nu(x_0)=r$ for any minimal element $\nu$ of $\TTTTT(P)$.
\end{cor}

%
%

Next we state some basic facts on $\TTTTT(P)$.

\begin{lemma}
\mylabel{lem:ideal red}
Let $I$ be a non-empty poset ideal of $P$ and $\nu\in\TTTTT(P)$.
Suppose that $\nu(x)-\nu(y)\geq 2$ for any $x\in I$ and $y\in P^+\setminus I$ with $x<y$.
If we set
$$
\nu'(x)=
\left\{
\begin{array}{ll}
\nu(x)&\quad\mbox{ if $x\not\in I$ and}\\
\nu(x)-1&\quad\mbox{ if $x\in I$,}
\end{array}
\right.
$$
then $\nu'\in\TTTTT(P)$.
In particular, $\nu$ is not a minimal element of $\TTTTT(P)$.
\end{lemma}
As a corollary, we have the following

\begin{lemma}
\mylabel{lem:use ideal red}
Let $\nu\in\TTTTT(P)$.
If $\{z\in P^+\mid \nu(z)>\rank[z,\infty]\}$ is a non-empty poset ideal of $P^+$,
then $\nu$ is not a minimal element of $\TTTTT(P)$.
\end{lemma}

We also state the following fact.

\begin{lemma}
\mylabel{lem:3 sum pure}
Let $Q$ be a poset and $s$ a positive integer.
If for any $z_0$, $z_1$, $z_2$ and $z_3\in Q$ such that $z_0\leq z_1\leq z_2\leq z_3$,
$z_0$ is a minimal element of $Q$ and $z_3$ is a maximal element of $Q$, it holds
$$
\rank[z_0,z_1]+\rank[z_1,z_2]+\rank[z_2,z_3]=s,
$$
then $Q$ is pure of rank $s$.
\end{lemma}
\begin{proof}
Let
$
x_0<x_1<\cdots<x_t
$
be an arbitrary maximal chain of $Q$.
We shall show that $t=s$.

Since
$$
\rank[x_0,x_t]=\rank[x_0,x_0]+\rank[x_0,x_0]+\rank[x_0,x_t]=s
$$
by assumption,
we see that
$$
t\leq \rank[x_0,x_t]=s.
$$

Assume $t<s$.
Then 
$\{i\mid i<\rank[x_0,x_i]\}\neq\emptyset$.
Set $j=\min\{i\mid i<\rank[x_0,x_i]\}$.
Then $j\geq 1$. Further, we see by assumption, that
\begin{eqnarray*}
s&=&\rank[x_0,x_{j-1}]+\rank[x_{j-1},x_j]+\rank[x_j,x_t]\\
&=&(j-1)+1+\rank[x_j,x_t].
\end{eqnarray*}
Since 
$\rank[x_0,x_j]+\rank[x_j,x_t]=
\rank[x_0,x_j]+\rank[x_j,x_j]+\rank[x_j,x_t]=s$, we see that
$\rank[x_j,x_t]=s-\rank[x_0,x_j]$.
Thus,
$$
s=j+(s-\rank[x_0,x_j])<j+s-j=s,
$$
by the definition of $j$.
This is a contradiction.
\end{proof}

\section{A criterion of levelness of a Hibi ring}
\mylabel{sec:level}

In this section, we give a necessary and sufficient condition for a Hibi ring to be level.
Recall that 
$H$ is a finite distributive lattice with unique minimal element $x_0$,
$P$ is the set of \joinirred\ elements of $H$ and
$r=\rank P^+$.
First we make the following 
\begin{definition}\rm
\mylabel{def:condn}
Let $y_1$, $x_1$, $y_2$, $x_2$, \ldots, $y_t$, $x_t$ be a 
(possibly empty) sequence of elements in $P$.
We say the sequence 
$y_1$, $x_1$, $y_2$, $x_2$, \ldots, $y_t$, $x_t$ satisfies the \condn\ 
if 
\begin{enumerate}
\item
\mylabel{item:cond n 0}
$x_1\neq x_0$,
\item\mylabel{item:cond n 1}
$y_1>x_1<y_2>x_2<\cdots<y_t>x_t$ and
\item\mylabel{item:cond n 2}
for any $i$, $j$ with $1\leq i<j\leq t$,
$y_i\not\geq x_j$.
\end{enumerate}
\end{definition}

\begin{nnremark}
\rm
A \scn\ may be an empty sequence, i.e.,  $t$ may be $0$.
\end{nnremark}
For a sequence with condtion N, we make the following

\begin{nndefinition}\rm
Let
$y_1$, $x_1$, $y_2$, $x_2$, \ldots, $y_t$, $x_t$ 
be a \scn.
We set
$$
r(y_1, x_1, \ldots, x_t, y_t)\define
\sum_{i=1}^t(\rank[x_{i-1},y_i]-\rank[x_i,y_i])+\rank[x_t,\infty],
$$
where we set an empty sum to be 0.
\end{nndefinition}

\begin{nnremark}\rm
For an empty sequence, we set $r()=\rank[x_0,\infty]=r$.
\end{nnremark}
Next we state a useful criterion of minimal property 
of an element of $\TTTTT(P)$.
First we state the following

\begin{lemma}
\mylabel{lem:min t suf}
Supposet that $\nu\in\TTTTT(P)$.
If there are elements $z_1$, $w_1$, $z_2$, $w_2$, \ldots, $z_s$, $w_s\in P$ such that
\begin{enumerate}
\item
$z_1>w_1<z_2>w_2<\cdots<z_s>w_s$ and
\item
$\nu(w_i)-\nu(z_{i+1})=\rank[w_i,z_{i+1}]$ for $0\leq i\leq s$,
where we set $w_0=x_0$ and $z_{s+1}=\infty$,
\end{enumerate}
then $\nu$ is a minimal element of $\TTTTT(P)$.
\end{lemma}

\begin{nnremark}\rm
$s$ may be $0$ in Lemma \ref{lem:min t suf},
i.e., if $\nu(x_0)=r$, then $\nu$ is a minimal element of $\TTTTT(P)$.
\end{nnremark}
{\bf Proof of Lemma \ref{lem:min t suf}}\ 
The case where $s=0$ is obvious.
Thus, we consider the case where $s>0$.
Assume the contrary and suppose that there is $\nu'\in\TTTTT(P)$ with $\nu'<\nu$.
Then $\nu(w_0)-\nu'(w_0)=(\nu-\nu')(w_0)>0$
since $\nu-\nu'\in\overline\TTTTT(P)$ and 
$\nu\neq\nu'$.
Set
$j=\max\{i\mid\nu'(w_i)<\nu(w_i)\}$.
Since 
$\nu(w_j)-\nu(z_{j+1})=\rank[w_j,z_{j+1}]\leq\nu'(w_j)-\nu'(z_{j+1})$
by assumption,
we see that
$$
\nu'(z_{j+1})<\nu(z_{j+1}).
$$
In particular, 
$j<s$.
Since $\nu-\nu'\in\overline\TTTTT(P)$,
$(\nu-\nu')(w_{j+1})\geq(\nu-\nu')(z_{j+1})$, i.e.,
$\nu(w_{j+1})-\nu'(w_{j+1})\geq\nu(z_{j+1})-\nu'(z_{j+1})>0$.
This contradicts to the definition of $j$.
\qed

Next we show a kind of converse of this lemma.

\begin{lemma}
\mylabel{lem:min t nec}
Let $\nu$ be a minimal element of $\TTTTT(P)$.
Then there exists a sequence of elements 
$y_1$, $x_1$, \ldots, $y_t$, $x_t$ with \condn\ such that
$$
\nu(x_i)-\nu(y_{i+1})=\rank[x_i,y_{i+1}]
$$
for $0\leq i\leq t$, where we set $y_{t+1}=\infty$.
\end{lemma}

\begin{nnremark}\rm
$y_1$, $x_1$, \ldots, $y_t$, $x_t$ may be an empty
sequence, i.e., $t$ may be $0$.
\end{nnremark}

\noindent
{\bf Proof of Lemma \ref{lem:min t nec}}\
Set 
$U_1\define\{y\in P^+\mid\rank[x_0,y]_{P^+}=\nu(x_0)-\nu(y)\}$,
$D_1\define\{x\in P^+\setminus U_1\mid\exists y\in U_1$ such that $y>x\}$,
$U_2\define\{y\in P^+\setminus(U_1\cup D_1)\mid\exists x\in D_1$ such that 
$x<y$ and  $\rank[x,y]_{P^+}=\nu(x)-\nu(y)\}$,
$D_2\define\{x\in P^+\setminus (U_1\cup D_1\cup U_2)\mid\exists y\in U_2$ such that $y>x\}$
and so on.
Since $P$ is a finite set, the procedure stops after finite steps,
i.e., $U_t=\emptyset$ or $D_t=\emptyset$ for some $t$.

It is proved by the induction on $s$ that
\[
U_1\cup D_1\cup U_2\cup D_2\cup\cdots\cup U_s\cup D_s
\]
is a poset ideal of $P^+$ for any $s$.
Therefore, 
$$
I=U_1\cup D_1\cup U_2\cup D_2\cup\cdots
$$
is a poset ideal of $P^+$.

Suppose that $\infty\not\in I$.
Then $P^+\setminus I\neq \emptyset$.
Moreover, if $z\in I$, $z'\in P^+\setminus I$ and $z<z'$, then $\nu(z)\geq\nu(z')+2$.
In fact, if $z\in D_s$ for some $s$, then 
$\nu(z)-\nu(z')>\rank[z,z']\geq 1$, since $z'\not\in U_{s+1}\cup D_s\cup U_s\cup\cdots\cup U_1$.
If $z\in U_s$ for some $s$, then there is $x\in D_{s-1}$ such that 
$\nu(x)-\nu(z)=\rank[x,z]$, where we set $D_0=\{x_0\}$.
Since $z'\not\in U_s\cup D_{s-1}\cup U_{s-1}\cup\cdots\cup U_1$, we see that
$\nu(x)-\nu(z')>\rank[x,z']$.
Therefore,
\begin{eqnarray*}
\nu(z)-\nu(z')&=&
(\nu(x)-\nu(z'))-(\nu(x)-\nu(z))\\
&>&\rank[x,z']-\rank[x,z]\geq\rank[z,z']\geq 1.
\end{eqnarray*}
Thus, by Lemma \ref{lem:ideal red}, we see that $\nu$ is not a minimal element
of $\TTTTT(P)$, contradicts the assumption.

Therefore, $\infty\in I$.
Take $t$ with $\infty\in U_{t+1}$.
Set $y_{t+1}=\infty$,
take $x_t\in D_t$ such that $x_t<y_{t+1}$ and $\rank[x_t,y_{t+1}]=\nu(x_t)-\nu(y_{t+1})$,
take $y_t\in U_t$ such that $y_t>x_t$, 
take $x_{t-1}\in D_{t-1}$ such that $x_{t-1}<y_{t}$ and 
$\rank[x_{t-1},y_{t}]=\nu(x_{t-1})-\nu(y_{t})$,
take $y_{t-1}\in U_{t-1}$ such that $y_{t-1}>x_{t-1}$ 
and so on.
Then it is easily verified that
$y_1$, $x_1$, \ldots, $y_t$, $x_t$ is a \scn\  and
$$
\nu(x_i)-\nu(y_{i+1})=\rank[x_i,y_{i+1}]
$$
for $0\leq i\leq t$.
\qed

Here we state the following fact.

\begin{lemma}
\mylabel{lem:scn finite}
There are only finitely many \sscn.
\end{lemma}
\begin{proof}
If $y_1$, $x_1$, \ldots, $y_t$, $x_t$ is a \scn,
then $x_i\leq y_i$  and $x_j\not\leq y_i$ for $j>i$.
Therefore, we see that $x_1$, \ldots, $x_t$ are distinct elements of $P$.
Since $P$ is a finite set, we see the result.
\end{proof}
Now we make the following

\begin{nndefinition}\rm
We set
$\rmax\define\max
\{r(y_1,x_1,\ldots,y_t,x_t)\mid y_1$, $x_1$, \ldots, $y_t$, $x_t$ 
is a \scn$\}$.
\end{nndefinition}
We note here the following fact.

\begin{nnremark}\rm
$r$ is not necessarily equal to
$
\min\{r(y_1,x_1,\ldots,y_t,x_t)\mid y_1$, $x_1$, \ldots, $y_t$, $x_t$ 
is a \scn$\}$.
\end{nnremark}
As a corollary of Lemma \ref{lem:min t nec}, we see the following fact.

\begin{cor}
\mylabel{cor:nu x0 range}
If $\nu$ is a minimal element of $\TTTTT(P)$, then
$r\leq \nu(x_0)\leq \rmax$.
\end{cor}
\begin{proof}
Since $\nu\in\TTTTT(P)$, we see that
$r\leq\nu(x_0)$.
On the other hand, by Lemma \ref{lem:min t nec}, we see that there is a sequence
$y_1$, $x_1$, \ldots, $y_t$, $x_t$ with \condn\ such that
$\nu(x_{i-1})-\nu(y_i)=\rank[x_{i-1},y_i]$ for $1\leq i\leq t+1$,
where we set $y_{t+1}=\infty$.
Thus,
\begin{eqnarray*}
\nu(x_0)&=&\sum_{i=1}^t(\nu(x_{i-1})-\nu(y_i)+\nu(y_i)-\nu(x_{i}))
+\nu(x_t)-\nu(y_{t+1})\\
&=&
\sum_{i=1}^t(\rank[x_{i-1},y_i]-(\nu(x_i)-\nu(y_{i})))
+\rank[x_t,y_{t+1}]\\
&\leq&
\sum_{i=1}^t(\rank[x_{i-1},y_i]-\rank[x_i,y_i])
+\rank[x_t,y_{t+1}]\\
&=&r(y_1,x_1,\ldots,y_t,x_t)\\
&\leq&\rmax.
\end{eqnarray*}
\end{proof}

Next we define two elements of $\TTTTT(P)$ defined by a sequence of elements
with \condn.

\begin{definition}\rm
\mylabel{def:nuup down}
Let $y_1$, $x_1$, \ldots, $y_t$, $x_t$ be a sequence of elements with \condn.
Set $y_{t+1}=\infty$.
We define
$$
\mudown_{(y_1,x_1,\ldots,y_t,x_t)}(y_i)\define
\sum_{k=i}^t(-\rank[x_k,y_k]+\rank[x_k,y_{k+1}])
$$
for $1\leq i\leq t+1$ and
$$
\nudown_{(y_1,x_1,\ldots,y_t,x_t)}(z)\define
\max\{\rank[z,y_i]+\mudown_{(y_1,x_1,\ldots,y_t,x_t)}(y_i)\mid
z\leq y_i\}
$$
for $z\in P^+$.
We also define
$$
\muup_{(y_1,x_1,\ldots,y_t,x_t)}(x_i)\define
\rmax+\sum_{k=1}^i(-\rank[x_{k-1},y_k]+\rank[x_k,y_{k}])
$$
for $0\leq i\leq t$ and
$$
\nuup_{(y_1,x_1,\ldots,y_t,x_t)}(z)\define
\min\{-\rank[x_i,z]+\muup_{(y_1,x_1,\ldots,y_t,x_t)}(x_i)\mid
x_i\leq z\}
$$
for $z\in P^+$.
\end{definition}
Here we note the following fact.

\begin{lemma}
\mylabel{lem:nuup down in tp}
Let $y_1$, $x_1$, \ldots, $y_t$, $x_t$ be a sequence of elements with \condn.
Then
$\nudown_{(y_1,x_1,\ldots,y_t,x_t)}$, $\nuup_{(y_1,x_1,\ldots,y_t,x_t)}\in\TTTTT(P)$.
\end{lemma}
\begin{proof}
Set 
$\nudown=\nudown_{(y_1,x_1,\ldots,y_t,x_t)}$, $\nuup=\nuup_{(y_1,x_1,\ldots,y_t,x_t)}$ 
 and $y_{t+1}=\infty$.
It is easily verified by the definition that if $z_1$, $z_2\in P^+$ and $z_1<z_2$, then
$$
\nudown(z_1)>\nudown(z_2)\quad\mbox{and}\quad
\nuup(z_1)>\nuup(z_2).
$$
Further, since
$\nudown(\infty)=
\max\{\rank[\infty,y_i]+\mudown_{(y_1,x_1,\ldots,y_t,x_t)}(y_i)\mid \infty\leq y_i\}=0$,
we see by the above inequality that $\nudown(z)>0$ for any $z\in P$.

Now we prove that $\nuup(z)>0$ for any $z\in P$.
It is enough to show that
$\nuup(\infty)=
\min\{-\rank[x_i,\infty]+\muup_{(y_1,x_1,\ldots,y_t,x_t)}(x_i)\mid x_i\leq \infty\}\geq 0$.
Assume the contrary and take $i$ 
with $-\rank[x_i,\infty]+\muup_{(y_1,x_1,\ldots, y_t,x_t)}(x_i)<0$.
Then $y_1$, $x_1$, \ldots, $y_i$, $x_i$ is a \scn\ and
$\rmax-r(y_1,x_1,\ldots, y_i,x_i)
=-\rank[x_i,\infty]+\muup_{(y_1,x_1,\ldots, y_t,x_t)}(x_i)<0$.
This contradicts to the definition of $\rmax$.
\end{proof}
We state the following important properties of $\nudown$ and $\nuup$.

\begin{lemma}
\mylabel{lem:nu up down min}
Let $y_1$, $x_1$, \ldots $y_t$, $x_t$ be a sequence of elements in $P$
with \condn.
Suppose that 
$r(y_1,x_1,\ldots, y_t,x_t)=\rmax$.
Then $\nudown_{(y_1, x_1, \ldots, y_t,x_t)}$ and
$\nuup_{(y_1, x_1, \ldots, y_t,x_t)}$ are minimal elements of $\TTTTT(P)$.
Further,
\begin{equation}
\nudown_{(y_1,x_1,\ldots, y_t,x_t)}(y_i)=
\mudown_{(y_1,x_1,\ldots, y_t,x_t)}(y_i),
\mylabel{eq:nu mu y}
\end{equation}
\begin{equation}
\nudown_{(y_1,x_1,\ldots, y_t,x_t)}(x_{i-1})=
\rank[x_{i-1},y_{i}]+\mudown_{(y_1,x_1,\ldots, y_t,x_t)}(y_{i}),
\mylabel{eq:nu mu x}
\end{equation}
\begin{equation}
\nuup_{(y_1,x_1,\ldots, y_t,x_t)}(x_{i-1})=
\muup_{(y_1,x_1,\ldots, y_t,x_t)}(x_{i-1}),
\mylabel{eq:nu mu up x}
\end{equation}
\begin{equation}
\nuup_{(y_1,x_1,\ldots, y_t,x_t)}(y_i)=
-\rank[x_{i-1},y_{i}]+\muup_{(y_1,x_1,\ldots, y_t,x_t)}(x_{i-1})
\mylabel{eq:nu mu up y}
\end{equation}
for $1\leq i\leq t+1$, where we set $y_{t+1}=\infty$.
In particular,
$\nudown_{(y_1, x_1, \ldots, y_t,x_t)}(x_0)=
\nuup_{(y_1, x_1, \ldots, y_t,x_t)}(x_0)=\rmax$.
\end{lemma}
\begin{proof}
By Lemma \ref{lem:nuup down in tp}, we see that 
$\nudown_{(y_1, x_1, \ldots, y_t,x_t)}\in\TTTTT(P)$.
First we show \refeq{eq:nu mu y}.
Assume the contrary and 
take $j$ with 
$\nudown_{(y_1,x_1,\ldots, y_t,x_t)}(y_j)\neq
\mudown_{(y_1,x_1,\ldots, y_t,x_t)}(y_j)$.
Then  $j\leq t$,
$\nudown_{(y_1,x_1,\ldots, y_t,x_t)}(y_j)>
\mudown_{(y_1,x_1,\ldots, y_t,x_t)}(y_j)$
and there exists $y_i$ such that $y_j\leq y_i$ and 
$$
\nudown_{(y_1,x_1,\ldots, y_t,x_t)}(y_j)=
\rank[y_j, y_i]+\mudown_{(y_1,x_1,\ldots, y_t,x_t)}(y_i).
$$
Since $y_1$, $x_1$, \ldots, $y_t$, $x_t$ is a \scn,
$y_k\not\geq x_j$ for any $k$ with $1\leq k\leq j-1$.
Thus, $i\geq j$, since $x_j<y_j\leq y_i$.
Since $x_{j-1}<y_j\leq y_i$, we see that 
$y_1$, $x_1$, \ldots, $y_{j-1}$, $x_{j-1}$, $y_i$, $x_i$, \ldots, $y_t$, $x_t$
is a \scn.
Further, since
$$
\rank[x_{j-1},y_i]\geq\rank[x_{j-1},y_j]+\rank[y_j,y_i],
$$
we see that
\begin{eqnarray*}
&&r(y_1,x_1, \ldots, y_{j-1},x_{j-1},y_i,x_i, \ldots,y_t, x_t)\\
&=&\sum_{k=1}^{j-1}(\rank[x_{k-1},y_k]-\rank[x_k,y_k])
+\rank[x_{j-1},y_i]+
\mudown_{(y_1,x_1,\ldots, y_t,x_t)}(y_i)\\
&\geq&
\sum_{k=1}^{j-1}(\rank[x_{k-1},y_k]-\rank[x_k,y_k])
\\&&\quad
+\rank[x_{j-1},y_j]+\rank[y_j,y_i]+
\mudown_{(y_1,x_1,\ldots, y_t,x_t)}(y_i)\\
&=&
\sum_{k=1}^{j-1}(\rank[x_{k-1},y_k]-\rank[x_k,y_k])
+\rank[x_{j-1},y_j]+
\nudown_{(y_1,x_1,\ldots, y_t,x_t)}(y_j)\\
&>&
\sum_{k=1}^{j-1}(\rank[x_{k-1},y_k]-\rank[x_k,y_k])
+\rank[x_{j-1},y_j]+
\mudown_{(y_1,x_1,\ldots, y_t,x_t)}(y_j)\\
&=&
r(y_1,x_1,\ldots,y_t,x_t)\\
&=&
\rmax.
\end{eqnarray*}
This contradicts to the definition of $\rmax$.
Therefore, we see \refeq{eq:nu mu y}.

Next we show \refeq{eq:nu mu x}.
Assume the contrary and take $j$ with
$\nudown_{(y_1,x_1,\ldots, y_t,x_t)}(x_j)\neq
\rank[x_j,y_{j+1}]+\mudown_{(y_1,x_1,\ldots, y_t,x_t)}(y_{j+1})$.
Then 
$\nudown_{(y_1,x_1,\ldots, y_t,x_t)}(x_j)>
\rank[x_j,y_{j+1}]+\mudown_{(y_1,x_1,\ldots, y_t,x_t)}(y_{j+1})$
and there is $y_i$ such that $x_j\leq y_i$ and
$$
\nudown_{(y_1,x_1,\ldots, y_t,x_t)}(x_j)=
\rank[x_j,y_{i}]+\mudown_{(y_1,x_1,\ldots, y_t,x_t)}(y_{i}).
$$
Since $y_1$, $x_1$, \ldots, $y_t$, $x_t$ is a \scn, we see that $i\geq j$.
Moreover, since
\begin{eqnarray*}
&&\rank[x_j,y_{j}]+\mudown_{(y_1,x_1,\ldots, y_t,x_t)}(y_{j})\\
&=&\rank[x_j,y_{j+1}]+\mudown_{(y_1,x_1,\ldots, y_t,x_t)}(y_{j+1})\\
&\neq&\nudown_{(y_1,x_1,\ldots, y_t,x_t)}(x_j),
\end{eqnarray*}
we see that $i\geq j+2$.
Therefore, 
$y_1$, $x_1$, \ldots, $y_j$, $x_j$, $y_i$, $x_i$, \ldots, $y_t$, $x_t$ is a \scn\
and
\begin{eqnarray*}
&&r(y_1,x_1,\ldots,y_j,x_j,y_i,x_i,\ldots, y_t, x_t)\\
&=&\sum_{k=1}^j(\rank[x_{k-1},y_k]-\rank[x_k,y_k])+\rank[x_j,y_i]+
\mudown_{(y_1,x_1, \ldots, y_t,x_t)}(y_i)\\
&=&\sum_{k=1}^j(\rank[x_{k-1},y_k]-\rank[x_k,y_k])+\nudown_{(y_1,x_1,\ldots, y_t,x_t)}(x_j)\\
&>&\sum_{k=1}^j(\rank[x_{k-1},y_k]-\rank[x_k,y_k])+\rank[x_j,y_{j+1}]
+\mudown_{(y_1,x_1,\ldots, y_t,x_t)}(y_{j+1})\\
&=&r(y_1,x_1,\ldots,y_t,x_t)\\
&=&\rmax,
\end{eqnarray*}
contradicting the definition of $\rmax$.
Therefore, we see \refeq{eq:nu mu x}.

By \refeq{eq:nu mu y}, \refeq{eq:nu mu x} and Lemma \ref{lem:min t suf},
we see that $\nudown_{(y_1,x_1, \ldots, y_t, x_t)}$ is a minimal element of
$\TTTTT(P)$.
Further, we see that
$\nudown_{(y_1,x_1,\ldots,y_t,x_t)}(x_0)=r(y_1,x_1,\ldots, y_t, x_t)=\rmax$.

By considering the poset $Q$ whose base set is $P^+$ and 
$$
z<w \mbox{ in $Q$}{\iff}z>w\mbox{ in $P^+$},
$$
we see that $\nuup_{(y_1,x_1, \ldots, y_t, x_t)}$ is also a minimal
element of $\TTTTT(P)$ and 
\refeq{eq:nu mu up x} and \refeq{eq:nu mu up y} hold.
\end{proof}

Since $\deg T^\nu=\nu(x_0)$ for $\nu\in\TTTTT(P)$,
we see by 
Corollaries \ref{cor:cri level} and \ref{cor:nu x0 range} and Lemma \ref{lem:nu up down min},
the following
%
\begin{thm}
\mylabel{thm:level main}
$\RRRRR_K(H)$ is level if and only if 
$r(y_1,x_1,\ldots, y_t,x_t)\leq r$
for any sequence of elements in $P$ with \condn,
i.e.,
$\rmax=r$.
\end{thm}
%
As a corollary,
we can reprove our previous result.
\begin{cor}[\cite{miy}]
If $[x,\infty]$ is pure for any $x\in P\setminus\{x_0\}$,
then $\RRRRR_K(H)$ is level.
\end{cor}
\begin{proof}
By assumption, 
\[
\rank[x,y]=\rank[x,\infty]-\rank[y,\infty]
\]
for any
$x$, $y\in P\setminus \{x_0\}$ with $x\leq y$.
Therefore,  for any sequence
$y_1$, $x_1$, \ldots, $y_t$, $x_t$ with \condn, 
\begin{eqnarray*}
&&r(y_1,x_1, \ldots, y_t,x_t)\\
&=&\sum_{i=1}^t(\rank[x_{i-1},y_i]-\rank[x_i,y_i])+\rank[x_t,\infty]\\
&=&\rank[x_0,y_1]+\rank[y_1,\infty]\\
&\leq&r.
\end{eqnarray*}
\end{proof}

Next we show that for any integer $d$ with $r\leq d\leq \rmax$,
there is a generator of the canonical module of $\RRRRR_K(H)$ with degree $d$.
First we state the following

\begin{lemma}
\mylabel{lem:min t red}
Let $\nu$ be a minimal element of $\TTTTT(P)$ and $k$ a positive integer.
Set
$
\nu_k(x)=\max\{\nu(x)-k,\rank[x,\infty]\}
$
for $x\in P^+$.
Then $\nu_k$ is a minimal element of $\TTTTT(P)$.
\end{lemma}
\begin{proof}
It is clear  that $\nu_k\in\TTTTT(P)$.
By Lemma \ref{lem:min t nec}, we see that there is a sequence
$y_1$, $x_1$, \ldots, $y_t$, $x_t$ of elements of $P$ with \condn\ such that
$$
\nu(x_i)-\nu(y_{i+1})=\rank[x_i,y_{i+1}]
$$
for $0\leq i\leq t$, where we set $y_{t+1}=\infty$.

First consider the case where
$
\nu_k(x_i)-\nu_k(y_{i+1})=\rank[x_i,y_{i+1}]
$
for $0\leq i\leq t$.
Then, by Lemma \ref{lem:min t suf}, we see that $\nu_k$ is a minimal element of $\TTTTT(P)$.
Next consider the case that there exists $i$ with
$
\nu_k(x_i)-\nu_k(y_{i+1})\neq\rank[x_i,y_{i+1}]
$.
Set
$$
j=\min\{i\mid
\nu_k(x_i)-\nu_k(y_{i+1})>\rank[x_i,y_{i+1}]\}.
$$
If $\nu_k(x_j)=\nu(x_j)-k$, then
\begin{eqnarray*}
&&
\nu_k(x_j)-\nu_k(y_{j+1})\\
&=&
\nu(x_j)-k-\max\{\nu(y_{j+1})-k,\rank[y_{j+1},\infty]\}\\
&\leq&
\nu(x_j)-\nu(y_{j+1})\\
&=&
\rank[x_j,y_{j+1}].
\end{eqnarray*}
This contradicts to the definition of $j$.

Thus, $\nu_k(x_j)=\rank[x_j,\infty]$.
Since 
$\nu_k(x_i)-\nu_k(y_{i+1})=\rank[x_i,y_{i+1}]$ for $0\leq i\leq j-1$
by the definition of $j$, by applying Lemma \ref{lem:min t suf} to
$y_1$, $x_1$, \ldots, $y_j$, $x_j$, we see that $\nu_k$ is a minimal element
of $\TTTTT(P)$.
\end{proof}
As a corollary, we see the following fact.

\begin{thm}
\mylabel{prop:gen deg}
Let $d$ be an integer with $r\leq d\leq \rmax$.
Then there exists a generator of the canonical module of 
$\RRRRR_K(H)$ with degree $d$.
\end{thm}
\begin{proof}
Set $k=\rmax-d$.
Take a sequence $y_1$, $x_1$, \ldots, $y_t$, $x_t$ with \condn\
such that $r(y_1,x_1,\ldots, y_t,x_t)=\rmax$ and put
$$
\nu(z)=\max\{\nudown_{(y_1,x_1,\ldots,y_t,x_t)}(z)-k,\rank[z,\infty]\}
$$
for $z\in P$.
Then $\nu(x_0)=\max\{\nudown_{(y_1,x_1,\ldots,y_t,x_t)}(x_0)-k,r\}
=\max\{\rmax-k,r\}=\max\{d,r\}=d$.
Since $\nu$ is a minimal element of $\TTTTT(P)$ by Lemmas \ref{lem:min t red}
and \ref{lem:nu up down min}, we see by Corollary \ref{cor:can gen}
that $T^\nu$ is a generator of the canonical module of $\RRRRR_K(H)$ with degree $d$.
\end{proof}


Finally in this section, we make a remark on $P$ when $\RRRRR_K(H)$ is level.
First we state the following

\begin{lemma}
\mylabel{lem:non float sum}
Set 
$
F=\{x\in P\mid \rank[x_0,x]_{P^+}+\rank[x,\infty]_{P^+}<r\}
$.
Suppose that $\RRRRR_K(H)$ is level.
Then for any $z_1$, $z_2\in P^+\setminus F$ with $z_1\leq z_2$, 
\begin{equation}
\rank[x_0,z_1]_{P^+}+\rank[z_1,z_2]_{P^+}+\rank[z_2,\infty]_{P^+}=r.
\mylabel{eq:3 sum r}
\end{equation}
and
\begin{equation}
\rank[z_1,z_2]_{P^+\setminus F}=\rank[z_1,z_2]_{P^+}.
\mylabel{eq:-f rank}
\end{equation}
\end{lemma}
\begin{proof}
We first show \refeq{eq:3 sum r}.
The cases where $z_1=x_0$, $z_2=\infty$ or $z_1=z_2$ are clear from the definition
of $F$.
Suppose that $z_1$, $z_2\in P\setminus F$ and $x_0< z_1<z_2$.
Then $z_2$, $z_1$ is a \scn.
Since $\RRRRR_K(H)$ is level, we see by Theorem \ref{thm:level main} that
$$
\rank[x_0,z_2]_{P^+}-\rank[z_1,z_2]_{P^+}+\rank[z_1,\infty]_{P^+}\leq r.
$$
Since $z_1$, $z_2\not\in F$, we see that 
$
\rank[x_0,z_2]_{P^+}=r-\rank[z_2,\infty]_{P^+}
$ and 
$
\rank[z_1,\infty]_{P^+}=r-\rank[x_0,z_1]_{P^+}
$.
Thus, 
$$
2r-\rank[x_0,z_1]_{P^+}-\rank[z_1,z_2]_{P^+}-\rank[z_2,\infty]_{P^+}\leq r,
$$
i.e.,
$$
\rank[x_0,z_1]_{P^+}+\rank[z_1,z_2]_{P^+}+\rank[z_2,\infty]_{P^+}\geq r.
$$
The opposite inequality is obvious.
Thus, we see \refeq{eq:3 sum r}.

Next we prove \refeq{eq:-f rank}.
Assume the contrary.
Take a maximal chain 
$$
z_1=w_0<w_1<\cdots<w_t=z_2
$$
of $[z_1,z_2]_{P^+}$ such that $t=\rank[z_1,z_2]_{P^+}$.
Since
\begin{eqnarray*}
&&\rank[x_0,w_i]_{P^+}+\rank[w_i,\infty]_{P^+}\\
&\geq&
\rank[x_0,z_1]_{P^+}+\rank[z_1,w_i]_{P^+}+\rank[w_i,z_2]_{P^+}+\rank[z_2,\infty]_{P^+}\\
&=&
\rank[x_0,z_1]_{P^+}+\rank[z_1,z_2]_{P^+}+\rank[z_2,\infty]_{P^+}\\
&=&r,
\end{eqnarray*}
we see that $w_i\not\in F$ for any $i=1$, \ldots, $t-1$.
Therefore,
$\rank[z_1,z_2]_{P^+\setminus F}\geq t=\rank[z_1,z_2]_{P^+}$.
The opposite inequality is obvious.
\end{proof}
By the above lemma and Lemma \ref{lem:3 sum pure}, we see the following fact.


\begin{prop}
\mylabel{prop:float del}
Let $F$ be as in Lemma \ref{lem:non float sum}.
If $\RRRRR_K(H)$ is level, then $P^+\setminus F$ is pure of rank $r$.
In particular, if $\RRRRR_K(H)$ is level and $F=\emptyset$, then 
$\RRRRR_K(H)$ is \gor.
\end{prop}

\begin{example}\rm
\mylabel{ex:float del pure}
Let 
\[
P_1=
\vcenter{\parindent0pt\hsize=40\unitlength
\begin{picture}(50,70)
\put(10,20){\circle*{2}}
\put(10,30){\circle*{2}}
\put(10,40){\circle*{2}}
\put(10,50){\circle*{2}}
\put(10,60){\circle*{2}}

\put(30,20){\circle*{2}}
\put(30,30){\circle*{2}}
\put(30,40){\circle*{2}}
\put(30,50){\circle*{2}}
\put(30,60){\circle*{2}}

\put(20,30){\circle*{2}}

\put(20,10){\circle*{2}}

\put(10,20){\line(0,1){40}}
\put(30,20){\line(0,1){40}}
\put(20,10){\line(0,1){20}}
\put(10,60){\line(1,-3){10}}
\put(20,30){\line(1,1){10}}

\put(10,20){\line(1,-1){10}}
\put(20,10){\line(1,1){10}}
\end{picture}
}
\quad\mbox{and}\quad
P_2=
\vcenter{\parindent0pt\hsize=40\unitlength
\begin{picture}(60,70)
\put(10,20){\circle*{2}}
\put(10,30){\circle*{2}}
\put(10,40){\circle*{2}}
\put(10,50){\circle*{2}}
\put(10,60){\circle*{2}}

\put(30,20){\circle*{2}}
\put(30,30){\circle*{2}}
\put(30,40){\circle*{2}}
\put(30,50){\circle*{2}}
\put(30,60){\circle*{2}}

\put(40,20){\circle*{2}}
\put(40,30){\circle*{2}}
\put(40,40){\circle*{2}}
\put(40,50){\circle*{2}}
\put(40,60){\circle*{2}}

\put(20,30){\circle*{2}}

\put(20,10){\circle*{2}}

\put(10,20){\line(0,1){40}}
\put(30,20){\line(0,1){40}}
\put(40,20){\line(0,1){40}}
\put(20,10){\line(0,1){20}}
\put(10,60){\line(1,-3){10}}
\put(20,30){\line(1,2){10}}
\put(20,30){\line(2,1){20}}

\put(10,20){\line(1,-1){10}}
\put(20,10){\line(1,1){10}}
\put(20,10){\line(2,1){20}}
\end{picture}
}
\quad.
\]
Then $\rank P_i^+=6$ for $i=1$, $2$.
Let $F_i$ be the subset of $P_i$ defined as in Lemma \ref{lem:non float sum} 
and let $H_i$ be the distributive lattice corresponding to
$P_i$ for $i=1$, $2$.
Then $P_i^+\setminus F_i$ is pure but $\RRRRR_K(H_i)$ is not level for $i=1$, $2$.
Thus, the converse of Proposition \ref{prop:float del} does not hold.
There are 2 generators with degree 6, 3 generators with degree 7 
and 6 generators with degree 8 of the 
canonical module of $\RRRRR_K(H_1)$
and
there are 2 generators with degree 6, 48 generators with degree 7 
and 108 generators with degree 8 of the 
canonical module of $\RRRRR_K(H_2)$.
\end{example}


\section{Characterization of level type 2 Hibi rings}
\mylabel{sec:level type 2}

As Example \ref{ex:float del pure} shows, it is very hard to describe \cm\ type
of the Hibi ring in terms of the combinatorial structure of $P$.
However, we can characterize 
Hibi ring 
$\RRRRR_K(H)$ to be of type 2 with respect to the combinatorial property of $P$.
Recall that, by Corollary \ref{cor:can gen}, we see that $\type\RRRRR_K(H)$ is 
the number of minimal elements of $\TTTTT(P)$.

In this section, we state a characterization of a Hibi ring to be level and of type 2.
First we make the following


\begin{nndefinition}\rm
Let 
$y_1$, $x_1$, $y_2$, $x_2$, \ldots, $y_t$, $x_t$ be a sequence of elements in $P$.
If the following 4 conditions are satisfied, we say that
$y_1$, $x_1$, $y_2$, $x_2$, \ldots, $y_t$, $x_t$ is an \shseq.
\begin{enumerate}
\item
$y_1$, $x_1$, $y_2$, $x_2$, \ldots, $y_t$, $x_t$ satisfies \condn.
\item
$r(y_1,x_1,\ldots,y_t,x_t)=\rmax$.
\item
If
$y'_1$, $x'_1$, $y'_2$, $x'_2$, \ldots, $y'_{t'}$, $x'_{t'}$ is a
sequence of elements in $P$ with \condn\  and
$r(y'_1,x'_1,\ldots,y'_{t'},x'_{t'})=\rmax$, 
then $t\leq t'$.
\item
For any $i$ with $1\leq i\leq t$,
$\rank[z,y_{i+1}]-\rank[z,y_i]
<\rank[x_i,y_{i+1}]-\rank[x_i,y_i]$
for any $z\in (x_i,y_{i+1}]\cap (x_i,y_i]$
and
$\rank[x_{i-1},z]-\rank[x_i,z]
<\rank[x_{i-1},y_{i}]-\rank[x_i,y_i]$
for any $z\in [x_{i-1},y_{i})\cap [x_i,y_i)$,
where we set $y_{t+1}=\infty$.
\end{enumerate}
\end{nndefinition}
It is clear that there exists an \irseq.
Further, 
by Theorem \ref{thm:level main},
$\RRRRR_K(H)$ is level if and only if the empty sequence is an \shseq.

Next we state the following

\begin{lemma}
\mylabel{lem:number float}
Set
$F=\{x\in P\mid \rank[x_0,x]+\rank[x,\infty]<r\}$.
Then the number of generators of the canonical module of degree $r$ is
greater than $\#F$.
In particular, $\type\RRRRR_K(H)>\#F$.
\end{lemma}
\begin{proof}
Set $F=\{f_1, f_2, \ldots, f_u\}$ and 
$i<j$ if $f_i<f_j$.
For $t$ with $1\leq t\leq u+1$, set
$$
\nu_t(x)=
\left\{
\begin{array}{ll}
\rank[x,\infty]&\quad\mbox{if $x\not\in F$ or $x\in F$ and $x=f_j$ with $j\geq t$,}\\
r-\rank[x_0,x]&\quad\mbox{otherwise.}
\end{array}
\right.
$$
Then it is easily verified that $\nu_t$ is an  element of $\TTTTT(P)$.
Further, since $\nu_t(x_0)=r$, we see that $\nu_t$ is a minimal element of $\TTTTT(P)$.

Since for any $t$, $t'$ with $t<t'$,
$\nu_t(f_t)<\nu_{t'}(f_t)$,
we see that $\nu_t\neq\nu_{t'}$.
Therefore, we see that there are at least $u+1$ minimal elements $\nu$ of 
$\TTTTT(P)$ such that $\nu(x_0)=r$.
\end{proof}

Now we state the main theorem of this section.

\begin{thm}
\mylabel{thm:level type 2}
$\RRRRR_K(H)$ is level and $\type \RRRRR_K(H)=2$ 
if and only if there exists $z\in P$ with the following conditions. 
\begin{enumerate}
\item
\mylabel{item:rank sum}
$\rank[x_0,z]+\rank[z,\infty]=r-1$.
\item
\mylabel{item:setminus pure}
$P^+\setminus \{z\}$ is pure of rank $r$.
\item
\mylabel{item:int pure}
$[x_0,z]$ and $[z,\infty]$ are pure.
\end{enumerate}
\end{thm}
\begin{nnremark}\rm
As Example \ref{ex:float del pure} shows, \ref{item:int pure} of 
Theorem \ref{thm:level type 2} does not follow from \ref{item:rank sum}
and \ref{item:setminus pure}.
\end{nnremark}
%
%
{\bf Proof of Theorem \ref{thm:level type 2}}\ 
We first assume that $\RRRRR_K(H)$ is level and $\type\RRRRR_K(H)=2$
and prove that
there exists $z\in P$ with 
\ref{item:rank sum},
\ref{item:setminus pure} and
\ref{item:int pure}.
Set 
$F=\{x\in P\mid\rank[x_0,x]+\rank[x,\infty]<r\}$.
If $F=\emptyset$ then, by Proposition \ref{prop:float del}, we see that 
$\RRRRR_K(H)$ is \gor, contradicting the assumption.
Therefore, $F\neq\emptyset$.
If $\#F\geq 2$, then by Lemma \ref{lem:number float}, we see that 
$\type \RRRRR_K(H)>2$, again contradicting the assumption.
Thus, $\#F=1$.

Set $F=\{z\}$.
We show that $z$ satisfies \ref{item:rank sum}, \ref{item:setminus pure} and 
\ref{item:int pure}.
Suppose that 
$$
\rank[x_0,z]+\rank[z,\infty]\leq r-2.
$$
Then if we set
$$
\mu_t(x)=
\left\{
\begin{array}{ll}
\rank[x,\infty]&\quad\mbox{if $x\neq z$,}\\
\rank[x,\infty]+t&\quad\mbox{if $x=z$,}
\end{array}
\right.
$$
then $\mu_t$ is a minimal element of $\TTTTT(P)$ for $0\leq t\leq 2$.
Thus, $\type\RRRRR_K(H)\geq 3$, contradicting the assumption.

Therefore,
$$
\rank[x_0,z]+\rank[z,\infty]=r-1,
$$
i.e., we see \ref{item:rank sum}.
Further, we see by Proposition \ref{prop:float del}, that $P^+\setminus\{z\}$ is
pure of rank $r$.
Thus, we see \ref{item:setminus pure}.

Now let $y$ be an arbitrary element of $[z,\infty]$ such that
$z\covered y$.
We shall show that $[y,\infty]$ is pure and $\rank[y,\infty]=\rank[z,\infty]-1$.
The case where $y=\infty$ is clear.
Suppose $y\neq\infty$.
Then $y$, $z$ is a \scn.
Since $\RRRRR_K(H)$ is level, we see by Theorem \ref{thm:level main} that
$
\rank[x_0,y]-1+\rank[z,\infty]=r(y,z)\leq r
$.
Since $\rank[x_0,y]=r-\rank[y,\infty]$, we see that
$
\rank[z,\infty]-\rank[y,\infty]\leq 1$.
Therefore, 
$\rank[y,\infty]=\rank[z,\infty]-1$.
Further,  $[y,\infty]$ is pure by Proposition \ref{prop:float del}.

Since $y$ is an arbitrary element of 
$[z,\infty]$ with $z\covered y$, we see 
that 
$[z,\infty]$ is pure.
We see that $[x_0,z]$ is pure by the same way.
Thus, we see \ref{item:int pure}.

Next we assume that
there exists $z\in P$ with 
\ref{item:rank sum},
\ref{item:setminus pure} and
\ref{item:int pure}
and prove that $\RRRRR_K(H)$ is level and $\type\RRRRR_K(H)=2$.

We first note that 
it follows from \ref{item:setminus pure} that
for any $w_1$, $w_2\in P^+\setminus\{z\}$,
$\rank[w_1,w_2]_{P^+\setminus\{z\}}=\rank[w_1,w_2]_{P^+}$.
In particular, 
$\rank[x,y]=
\rank[x_0,y]-\rank[x_0,x]$
for any $x$, $y\in P^+\setminus\{z\}$ with $x<y$. 
We also see 
that if $x<z$, then
$
\rank[x,z]=\rank[x_0,z]-\rank[x_0,x]$
and if $y>z$, then
$
\rank[z,y]=\rank[z,\infty]-\rank[y,\infty]
$,
since $[x_0,z]$ and $[z,\infty]$ are pure.

Now let 
$y_1$, $x_1$, \ldots, $y_t$, $x_t$ be an \irseq\ and set $y_{t+1}=\infty$.
We shall show that $t=0$.
Assume the contrary.
Then, since $y_1$, $x_1$, \ldots, $y_t$, $x_t$ is an \irseq,
we see that $r(y_2,x_2, \ldots, y_t,x_t)<r(y_1, x_1,\ldots, y_t,x_t)$,
i.e., 
\begin{equation}
\mylabel{eq:irred first}
\rank[x_0,y_2]<\rank[x_0,y_1]-\rank[x_1,y_1]+\rank[x_1,y_2].
\end{equation}

First consider the case where $x_1\neq z$.
Since $\rank[x_1,y_i]=\rank[x_0,y_i]-\rank[x_0,x_1]$ for $i=1$, $2$,
we see that the right hand side of \refeq{eq:irred first} is equal to
$\rank[x_0,y_2]$.
This is a contradiction.

Next consider the case where $x_1=z$.
Since $\rank[x_1,y_i]=\rank[x_1,\infty]-\rank[y_i,\infty]$ for $i=1$, $2$
we see that the right hand side of \refeq{eq:irred first} is equal to
$$
\rank[x_0,y_1]+\rank[y_1,\infty]-\rank[y_2,\infty].
$$
Since $y_i\neq z$ for $i=1,2$, we see that
$\rank[y_i,\infty]=r-\rank[x_0,y_i]$ for $i=1,2$.
Therefore, the right hand side of \refeq{eq:irred first} is equal to
$\rank[x_0,y_2]$.
This is also a contradiction.
Thus, we see that $t=0$ and $\RRRRR_K(H)$ is level.

Let $\nu$ be an arbitrary minimal element of $\TTTTT(P)$.
Since $\RRRRR_K(H)$ is level, we see that $\nu(x_0)=r$.
Thus, we see that $\nu(x)\geq\rank[x,\infty]$ and
$r-\nu(x)=\nu(x_0)-\nu(x)\geq\rank[x_0,x]$
i.e.,
$\rank[x,\infty]\leq\nu(x)\leq r-\rank[x_0,x]$
for any $x\in P$.
In particular, we see by \ref{item:setminus pure} that
$\nu(x)=\rank[x,\infty]$ for any $x\in P^+\setminus\{z\}$.
We also see by \ref {item:rank sum} that
$$
\rank[z,\infty]\leq\nu(z)\leq r-\rank[x_0,z]=\rank[z,\infty]+1.
$$
Thus $\type\RRRRR_K(H)\leq 2$.
Further, we see $\type \RRRRR_K(H)\geq 2$ by Lemma \ref{lem:number float}.
\qed


\section{Characterization of non-level type 2 Hibi rings}
\mylabel{sec:non level type 2}

In this final section, we give a characterization of 
a Hibi ring to be 
non-level and of type 2.


\begin{thm}
\mylabel{thm:non level type 2}
$\RRRRR_K(H)$ is non-level and $\type\RRRRR_K(H)=2$ if and only if
there exist $x$, $y\in P\setminus\{x_0\}$ with the following conditions.
\begin{enumerate}
\item
\mylabel{item:cover}
$x\covered y$.
\item
\mylabel{item:r+2}
$\rank[x_0,y]+\rank[x,\infty]=r+2$.
\item
\mylabel{item:union}
$P^+=[x_0,y]\cup[x,\infty]$.
\item
\mylabel{item:rank sum 2}
$\rank[x_0,z_1]+\rank[z_1,z_2]+\rank[z_2,\infty]=r$ for any $z_1$, $z_2\in P^+$
with $z_1\leq z_2$ and $(z_1,z_2)\neq (x,y)$.
\end{enumerate}
\end{thm}
\begin{proof}
Set $F=\{z\in P\mid \rank[x_0,z]+\rank[z,\infty]<r\}$.

First we assume that $\RRRRR_K(H)$ is non-level and $\type\RRRRR_K(H)=2$
and prove that there exist $x$, $y\in P$ with conditions \ref{item:cover} to \ref{item:rank sum 2}.
Since $\RRRRR_K(H)$ is not level, we see by Theorem \ref{thm:level main} that $\rmax>r$.
Further, since $\type\RRRRR_K(H)\geq\rmax-r+1$ by Theorem \ref{prop:gen deg}, we see 
that $\rmax=r+1$ and the generating system of the canonical module of $\RRRRR_K(H)$ 
consists of $2$ elements: one has degree $r$ and the other one has degree $r+1$.
Thus, by Lemma \ref{lem:number float}, we see that $F=\emptyset$.

Let $y_1$, $x_1$, \ldots, $y_t$, $x_t$ be an \irseq\
and set $y_{t+1}=\infty$.
First we claim that $t=1$.
Since $\rmax>r$, we see that $t\geq 1$.
Further, since 
$y_1$, $x_1$, \ldots, $y_t$, $x_t$ is an \shseq,
we see that
$r(y_2,x_2, \ldots, y_t,x_t)<r(y_1,x_1,\ldots,y_t,x_t)$,
i.e.,
$$
\rank[x_0,y_2]<\rank[x_0,y_1]-\rank[x_1,y_1]+\rank[x_1,y_2].
$$
Moreover, since $F=\emptyset$, we see that
\begin{eqnarray*}
r&=&\rank[x_0,y_2]+\rank[y_2,\infty]\\
&<&\rank[x_0,y_1]-\rank[x_1,y_1]+\rank[x_1,y_2]+\rank[y_2,\infty]\\
&\leq&\rank[x_0,y_1]-\rank[x_1,y_1]+\rank[x_1,\infty]\\
&=&r(y_1,x_1).
\end{eqnarray*}
Since $\rmax=r+1$ and $y_1$, $x_1$, \ldots, $y_t$, $x_t$ is an \shseq,
we see that $r(y_1,x_1)=r+1$ and $t=1$.

Set $x=x_1$ and $y=y_1$.
We show that these $x$ and $y$ satisfy \ref{item:cover} to \ref{item:rank sum 2}.
We first show that \ref{item:cover}.
Assume the contrary.
Then 
there is $w\in P$ such that $x<w<y$ and
$
\rank[x,y]=\rank[x,w]+\rank[w,y]$.
Since $y,x$ is an \irseq, we see that
$$
\rank[x_0,w]-\rank[x,w]\leq\rank[x_0,y]-\rank[x,y]-1
$$
and
$$
-\rank[w,y]+\rank[w,\infty]\leq-\rank[x,y]+\rank[x,\infty]-1.
$$
Therefore,
\begin{eqnarray*}
&&\rank[x_0,w]-\rank[x,w]-\rank[w,y]+\rank[w,\infty]\\
&\leq&\rank[x_0,y]-2\rank[x,y]+\rank[x,\infty]-2\\
&=&r(y,x)-\rank[x,y]-2.
\end{eqnarray*}
On the other hand,
since
$$
\rank[x,w]+\rank[w,y]=\rank[x,y]
\quad\mbox{and}\quad
\rank[x_0,w]+\rank[w,\infty]=r,
$$
we see
$r\leq r(y,x)-2$.
This contradicts to the fact that $r(y,x)=r+1$.
Therefore, we see \ref{item:cover}.
Moreover, since
\begin{eqnarray*}
&&\rank[x_0,y]+\rank[x,\infty]-1\\
&=&\rank[x_0,y]-\rank[x,y]+\rank[x,\infty]\\
&=&r(y,x)\\
&=&r+1,
\end{eqnarray*}
we see \ref{item:r+2}.

Next we prove \ref{item:union}.
Assume the contrary and let $z\in P^+\setminus([x_0,y]\cup[x,\infty])$.
Since $z\not\leq y$, we see by the definition of $\nudown_{(y,x)}$ 
(see Definition \ref{def:nuup down}) that
$\nudown_{(y,x)}(z)=\rank[z,\infty]$.
On the other hand, since $z\not\geq x$, we see by the definition of $\nuup_{(y,x)}$
that $\nuup_{(y,x)}(z)=r+1-\rank[x_0,z]$.
Since the minimal generating system of the canonical module of $\RRRRR_K(H)$  
contains only one element with degree $r+1$
and $\nudown_{(y,x)}$ and $\nuup_{(y,x)}$ are minimal elements of $\TTTTT(P)$ with
degree $r+1$ by Lemma \ref{lem:nu up down min},
we see that
$\nudown_{(y,x)}=\nuup_{(y,x)}$.
Therefore,
$$
\rank[z,\infty]=\nudown_{(y,x)}(z)=\nuup_{(y,x)}(z)=r+1-\rank[x_0,z].
$$
This contradicts to the fact that $\rank P^+=r$.
Therefore, we see \ref{item:union}.

Finally, we prove \ref{item:rank sum 2}.
The case where $z_1=z_2$, $z_1=x_0$ or $z_2=\infty$ are clear since $F=\emptyset$.
Suppose that $x_0<z_1<z_2<\infty$.
Then $z_2$, $z_1$ is a \scn.
If $r(z_2,z_1)\leq r$, then
$$
\rank[x_0,z_2]-\rank[z_1,z_2]+\rank[z_1,\infty]\leq r.
$$
Since $F=\emptyset$, we see that
$$
2r-\rank[z_2,\infty]-\rank[z_1,z_2]-\rank[x_0,z_1]\leq r.
$$
Therefore,
$$
\rank[x_0,z_1]+\rank[z_1,z_2]+\rank[z_2,\infty]\geq r.
$$
Since the converse inequality holds in general, we see the result.

Now assume that $r(z_2,z_1)>r$.
Since $\rmax=r+1$, we see that $r(z_2,z_1)=\rmax=r+1$.
Moreover, since the generating system of the canonical module 
contains exactly one element with degree $r+1$,
we see that
$$
\nudown_{(y,x)}=\nuup_{(y,x)}=\nudown_{(z_2,z_1)}=\nuup_{(z_2,z_1)}
$$
by Lemma \ref{lem:nu up down min}.
Since
$
\nudown_{(z_2,z_1)}(y)=
\nudown_{(y,x)}(y)>
\rank[y,\infty]
$,
we see that $y\leq z_2$.
On the other hand, since
$\nudown_{(y,x)}(z_2)=\nudown_{(z_2,z_1)}(z_2)
>\rank[z_2,\infty]$,
we see that
$z_2\leq y$.
Therefore, we see that $z_2=y$.
We also see that $z_1=x$ by the same way.
This proves \ref{item:rank sum 2}.

Next we assume that 
there exist $x$, $y\in P$ with conditions \ref{item:cover} to \ref{item:rank sum 2}
and prove that $\RRRRR_K(H)$ is non-level and $\type\RRRRR_K(H)=2$.
First we see by \ref{item:rank sum 2} that 
$F=\emptyset$.
Further, we see 
that $y$, $x$ is a \scn\ and 
by \ref{item:cover} and \ref{item:r+2} 
that
$$
r(y,x)=\rank[x_0,y]-\rank[x,y]+\rank[x,\infty]=r+1.
$$
Thus, we see that $\rmax>r$.
In particular, $\RRRRR_K(H)$ is not level by Theorem \ref{thm:level main}.

Let $y_1$, $x_1$, \ldots, $y_t$, $x_t$ be an \shseq.
Since $\rmax>r$, we see that $t\geq 1$.
First consider the case where there is no $i$ such that $(x_i,y_i)=(x,y)$.
Then $r(y_1,x_1,\ldots, y_t,x_t)\leq r$ by \ref{item:cover}, \ref{item:r+2} and  
\ref{item:rank sum 2},
contradicting the fact that $y_1$, $x_1$, \ldots, $y_t$, $x_t$ is an
\irseq.
Thus, there exists $i$ such that $(x_i,y_i)=(x,y)$.
Then by \ref{item:rank sum 2}, we see that
$r(y_1,x_1,\ldots, y_t,x_t)=r(y_i,x_i)=r(y,x)=r+1$.
Therefore, we see that $y$, $x$ is the unique \irseq\ and $\rmax=r(y,x)=r+1$.

Let $\nu$ be a minimal element of $\TTTTT(P)$ with $\nu(x_0)=r$.
Since $\nu(z)\geq\rank[z,\infty]$ and $r-\nu(z)=\nu(x_0)-\nu(z)\geq\rank[x_0,z]$,
we see that $\nu(z)=\rank[z,\infty]$ for any $z\in P^+$, since $F=\emptyset$.
Therefore, in the generating system of the canonical module of $\RRRRR_K(H)$,
there exists exactly one element with degree $r$.

Let $\nu$ be a minimal element of $\TTTTT(P)$ with 
$
\nu(x_0)=r+1
$.
It is enough to
show that $\nu=\nudown_{(y,x)}$.
As in the proof of Lemma \ref{lem:min t nec}, set
$U_1\define\{w\in P^+\mid\rank[x_0,w]_{P^+}=\nu(x_0)-\nu(w)\}$,
$D_1\define\{z\in P^+\setminus U_1\mid\exists w\in U_1$ such that $w>z\}$,
$U_2\define\{w\in P^+\setminus(U_1\cup D_1)\mid\exists z\in D_1$ such that 
$z<w$ and $\rank[z,w]_{P^+}=\nu(z)-\nu(w)\}$,
$D_2\define\{z\in P^+\setminus (U_1\cup D_1\cup U_2)\mid\exists w\in U_2$ such that $w>z\}$
and so on.
Since $\nu(x_0)=r+1$, we see that $\infty\not\in U_1$.
On the other hand, since
$$
\infty\in U_1\cup D_1\cup U_2\cup D_2\cup\cdots,
$$
we see that $D_1\neq \emptyset$.

Let $z_1\in D_1$ and let $z_2\in U_1$ such that $z_2>z_1$.
Then 
\begin{eqnarray}
&&\rank[x_0,z_2]\nonumber\\
&=&\nu(x_0)-\nu(z_2)\nonumber\\
&=&\nu(x_0)-\nu(z_1)+\nu(z_1)-\nu(z_2)\nonumber\\
&\geq&\rank[x_0,z_1]+1+\rank[z_1,z_2].
\mylabel{eq:z1}
\end{eqnarray}
Therefore, we see by \ref{item:rank sum 2} that $(z_1,z_2)=(x,y)$.
Further, by \ref{item:cover}, \ref{item:r+2} and 
the fact that $F=\emptyset$,
we see that equality holds for \refeq{eq:z1}.
Therefore, we see that
\begin{equation}
\nu(x_0)-\nu(x)=\rank[x_0,x]+1
\mylabel{eq:nu x}
\end{equation}
and
\begin{equation}
\nu(x)-\nu(y)=1.
\mylabel{eq:nu x nu y}
\end{equation}

Let $z$ be an arbitrary element of $P$.
We shall show that $\nu(z)=\nudown_{(y,x)}(z)$.

First consider the case where $z=x$.
Since 
$\nu(x)=\rank[x,\infty]$ by \refeq{eq:nu x} and the
facts that $F=\emptyset$ and $\nu(x_0)=r+1$, 
we see that $\nu(x)=\nudown_{(y,x)}(x)$ by Lemma \ref{lem:nu up down min}.

Next consider the case where $z=y$.
Since $\nu(x)=\rank[x,\infty]$, we see by \refeq{eq:nu x nu y} and
\ref{item:cover} that
$\nu(y)=\rank[x,\infty]-\rank[x,y]=\nudown_{(y,x)}(y)$
by Lemma \ref{lem:nu up down min}.

Now consider the case where $z\leq y$ and $z\neq x$.
Since $\rank[x_0,y]=\rank[x_0,z]+\rank[z,y]$ by \ref{item:rank sum 2},
we see that $z\in U_1$ and $\nu(z)-\nu(y)=\rank[z,y]$.
Since $\nu(x_0)=r+1$,
we see that
$\nu(z)=\nu(x_0)-\rank[x_0,z]>\rank[z,\infty]$.
Therefore,
we see by the definition of $\nudown_{(y,x)}$
that
$$
\nu(z)=\nu(y)+\rank[z,y]=\nudown_{(y,x)}(y)+\rank[z,y]=\nudown_{(y,x)}(z).
$$

Finally consider the case where $z\not\leq y$.
By \ref{item:union}, we see that $z>x$.
Further, we see by \ref{item:rank sum 2} that
$\rank[x,\infty]=\rank[x,z]+\rank[z,\infty]$.
Since $\nu(x)=\rank[x,\infty]$, we see that
$\nu(z)=\rank[z,\infty]=\nudown_{(y,x)}(z)$.
\end{proof}

\end{document}